\documentclass[leqno,twoside,11pt]{amsart}

\usepackage{latexsym,esint,url}

\theoremstyle{plain}
\def\endproof{\hspace*{\fill}\mbox{\ \rule{.1in}{.1in}}\medskip }

\newtheorem{theorem}{Theorem}[section]

\newtheorem{lemma}[theorem]{Lemma}

\theoremstyle{definition}

\newtheorem{remark}[theorem]{Remark}

\numberwithin{equation}{section}
\numberwithin{figure}{section}

\begin{document}

\title[Convergence of equilibria of thin elastic shells]
{A note on convergence of low energy critical points \\
of nonlinear elasticity functionals,\\ 
for thin shells of arbitrary geometry}
\author{Marta Lewicka}
\address{Marta Lewicka, University of Minnesota, Department of Mathematics, 
206 Church St. S.E., Minneapolis, MN 55455}
\email{lewicka@math.umn.edu}

\date{\today}
\begin{abstract} 
We prove that the critical points of the $3$d nonlinear elasticity functional
on shells of small thickness $h$ and around the mid-surface $S$ of 
arbitrary geometry, converge as $h\to 0$
to the critical points of the von
K\'arm\'an functional on $S$, recently derived in \cite{lemopa1}.
This result extends the statement in \cite{MuPa}, derived for the case 
of plates when $S\subset\mathbb{R}^2$.
We further prove the same convergence result for the
weak solutions to the static equilibrium equations (formally the Euler-
Lagrange equations associated to the elasticity functional).
The convergences hold provided the elastic energy of the $3$d deformations scale 
like $h^4$ and the external body forces scale like $h^3$.
\end{abstract}

\maketitle
\tableofcontents

\section{Introduction and statement of the main results} 

Since the beginning of research in nonlinear elasticity, a major topic has 
been the derivation of lower dimensional theories,
appropriately approximating the three
dimensional theory on structures which are thin in one or more directions
(such as beams, rods, plates or shells).
Recently, the application of variational methods, notably the 
$\Gamma$-convergence \cite{dalmaso}, lead to many significant and rigorous results
in this setting \cite{LR1, FJMhier}. 
Roughly speaking, a $\Gamma$-limit approach guarantees the 
convergence of minimizers of a sequence of functionals, to the minimizers of 
the limit. However, it does not usually imply convergence 
of the possibly non-minimizing critical points
(the equilibria) and hence other tools must be applied 
to study this problem. 

In this note, following works  \cite{MoMu, MoMuS, MuPa} 
in which beams, rods and plates were analyzed, we study critical points 
of the $3$d nonlinear elasticity functional 
on a thin shell of arbitrary geometry, in the von K\'arm\'an scaling regime. 
A $\Gamma$-convergence result in this framework was recently derived
in \cite{lemopa1}, providing the natural from the minimization point of view
generalization of the von K\'arm\'an functional \cite{FJMhier} to shells. 
In analogy with the analysis done in 
\cite{MuPa} for plates, we now proceed to establish
convergence of weak solutions to the (formal) Euler-Lagrange 
equations (\ref{EL3d}),
as well as  convergence of critical points of the $3$d energy functionals 
(\ref{elastic-En}), to the critical points of the functional obtained
in \cite{lemopa1}. As pointed out in \cite{Ball} Problem 5,  
in general it is still unknown whether these 
two definitions of equilibria are equivalent.

\medskip

We now introduce the basic framework for our results.
We consider a $2$-dimensional surface $S$ embedded in $\mathbb{R}^3$, which is
compact, connected, oriented, of class $\mathcal{C}^{1,1}$,
and with boundary $\partial S$ being the union of finitely many 
(possibly none) Lipschitz curves.
A family $\{S^h\}_{h>0}$ of shells of small thickness $h$
around $S$ is given through:
$$S^h = \{z=x + t\vec n(x); ~ x\in S, ~ -h/2< t < h/2\},\qquad 0<h<h_0.$$
By $\vec n(x)$ we denote the unit normal to $S$, by $T_x S$ the tangent space,
and $\Pi(x) = \nabla \vec n(x)$ is the shape operator on $S$
(the negative second fundamental form). 
The projection onto $S$ along $\vec n$ is denoted by $\pi$.
We assume that $h<h_0$, with $h_0$ sufficiently small
to have $\pi$ defined on each $S^h$.

To a deformation $u\in W^{1,2}(S^h,\mathbb{R}^3)$ we associate its elastic
energy (scaled per unit thickness):
\begin{equation}\label{elastic-En}
E^h(u) = \frac{1}{h}\int_{S^h} W(\nabla u).
\end{equation} 
The stored energy density $W:\mathbb{R}^{3\times 3}\longrightarrow [0,\infty]$ 
is assumed to be $\mathcal{C}^2$ in a neighborhood of $SO(3)$, and to satisfy
the following normalization, frame indifference and nondegeneracy conditions:
\begin{equation}\label{W-assump}
\begin{split}
\forall F\in\mathbb{R}^{3\times 3}\quad \forall R\in SO(3)\qquad
& W(R) = 0, \quad W(RF) = W(F), \\
& W(F)\geq C \mathrm{dist}^2(F, SO(3)) 
\end{split}
\end{equation}
(with a uniform constant $C>0$). 
Our objective is to describe the limiting behavior, as
$h\to 0$, of critical points $u^h$ to the following total energy functionals:
\begin{equation}\label{total-intro}
J^h(u) = E^h(u) - \frac{1}{h}\int_{S^h} f^hu,
\end{equation} 
subject to external forces $f^h$, where we assume that:
$$f^h(x+t\vec n) = h\sqrt{e^h} f(x) \det(\mbox{Id}+t\Pi)^{-1},\quad
f\in L^2(S,\mathbb{R}^3) \mbox{ and } \int_S f =0.$$
Above, $e^h$ is a given sequence of positive
numbers obeying a prescribed scaling law.
It can be shown \cite{FJMhier, lemopa1} that if $f^h$ scale like $h^\alpha$, then 
the minimizers $u^h$ of (\ref{total-intro}) satisfy
$E^h(u^h)\sim h^\beta$ with $\beta= \alpha$ if $0 \le \alpha \le 2$ and $\beta
= 2\alpha -2$ if $\alpha > 2$. 
Throughout this paper we shall assume that $\beta\geq 4$, or more generally:
\begin{equation}\label{scaling-intro}
\lim_{h\to 0} e^h/h^4 =\kappa < +\infty,
\end{equation}
which for $S\subset\mathbb{R}^2$ corresponds to the von K\'arm\'an 
and the purely linear theories of plates, 
derived rigorously in \cite{FJMhier}.

\medskip

In our recent paper \cite{lemopa1}, the $\Gamma$-limit of $1/e^h J^h$
has been identified in the scaling range corresponding to
(\ref{scaling-intro}), and for arbitrary surfaces $S$. 
It turns out that the elastic energy scaling $E^h(u^h)\leq C e^h$ implies
that on $S$ the deformations $u^h_{\mid S}$ must be close to some rigid motion
$\bar Qx+c$, and that the first order term in the expansion of 
$\bar Q^T(u^h_{\mid S}-c) -\mbox{id}$ 
with respect to $h$, is an element  $V$ of the 
class $\mathcal V$ of {\em infinitesimal isometries} on $S$ \cite{Spivak}. 
The space $\mathcal{V}$ consists of vector fields
$V\in W^{2,2}(S,\mathbb{R}^3)$ for whom there 
exists  a matrix field $A\in W^{1,2}(S,\mathbb{R}^{3\times 3})$ so that:
\begin{equation}\label{Adef-intro} 
\partial_\tau V(x) = A(x)\tau \quad \mbox{and} \quad  A(x)^T= -A(x) \qquad 
\forall {\rm{a.e.}} \,\, x\in S \quad \forall \tau\in T_x S.
\end{equation}
Equivalently, the change of metric on $S$ induced by the
deformation $\mbox{id} + h V$ is at most of order $h^2$, 
for each $V\in\mathcal{V}$.

When in (\ref{scaling-intro}) $\kappa=0$, the limiting total energy 
is given by:
\begin{equation}\label{I-intro}
J(V,\bar Q)= \frac{1}{24} \int_S 
\mathcal{Q}_2\left(x,(\nabla(A\vec{n}) - A\Pi)_{tan}\right)~\mbox{d}x
- \int_S f\cdot \bar QV ~\mbox{d}x,
\qquad \forall V\in\mathcal{V},~~ \bar Q\in SO(3).
\end{equation} 
The first term above measures the first order change 
in the second fundamental form $\Pi$ of $S$, produced by $V$. 
The quadratic forms $\mathcal{Q}_2(x,\cdot)$ are given as follows:
$$ \mathcal{Q}_2(x, F_{tan}) = \min\{\mathcal{Q}_3(\tilde F); ~~ (\tilde F -
F)_{tan} = 0\}, \qquad \mathcal{Q}_3(F) = D^2 W(\mbox{Id})(F,F).$$
The form $\mathcal{Q}_3$ is defined for all $F\in\mathbb{R}^{3\times 3}$, 
while $\mathcal{Q}_2(x,\cdot)$  for a given $x\in S$, is defined on tangential 
minors $F_{tan}$ of such matrices.    
Both forms depend only on the symmetric parts of their arguments
and are positive definite on the space of symmetric matrices \cite{FJMgeo}. 
In the weak formulation of the Euler-Lagrange equations of (\ref{I-intro})
one naturally encounters the linear operators $\mathcal{L}_3$ and 
$\mathcal{L}_2(x,\cdot)$, defined on matrix spaces $\mathbb{R}^{3\times 3}$ and
$\mathbb{R}^{2\times 2}$ respectively, given by:
$$\forall F\in \mathbb{R}^{3\times 3}\qquad
\mathcal{Q}_3(F) = \mathcal{L}_3 F: F \quad \mbox{ and }\quad
\mathcal{Q}_2(x,F_{tan}) = \mathcal{L}_2(x, F_{tan}): F_{tan}.$$

For $\kappa>0$, the $\Gamma$-limit (which is the generalization of 
the von K\'arm\'an functional \cite{FJMhier} to shells), 
contains also a stretching term, measuring the total second order change in the
metric of $S$:
\begin{equation}\label{vonKarman}
J^{vK}(V,B_{tan},\bar Q)= \frac{1}{2} 
\int_S \mathcal{Q}_2\left(x,B_{tan} - \frac{\kappa}{2} (A^2)_{tan}\right)
+ \frac{1}{24} \int_S \mathcal{Q}_2\left(x,(\nabla(A\vec n) -
  A\Pi)_{tan}\right)
- \int_S f\cdot \bar QV.
\end{equation}
It involves a symmetric matrix field $B_{tan}$ belonging to the 
{\em finite strain space}: 
$$ \mathcal{B} = \mbox{cl}_{L^2(S)}\Big\{\mathrm{sym }\nabla w^h; 
~~ w^h\in W^{1,2}(S,\mathbb{R}^3)\Big\}. $$ 
The two terms in (\ref{vonKarman}) correspond,
in appearing order, to the stretching and bending energies of a sequence of
deformations $v^h = \mbox{id} + hV + h^2 w^h$ of $S$,
which is induced by a first order displacement $V\in\mathcal{V}$ and second order
displacements $w^h$ satisfying $\lim_{h\to 0}\mbox{sym} \nabla w^h = B_{tan}$.
The crucial property  of (\ref{vonKarman}) is the one-to-one
correspondence between the minimizing sequences $u^h$ of the total energies
$J^h(u^h)$, and their approximations (modulo rigid
motions $\bar Qx+c$) given by $v^h$ as above with $(V, B_{tan}, \bar Q)$ minimizing
$J^{vK}$, or $(V,\bar Q)$ minimizing $J$ when $\kappa=0$. 

\medskip

The purpose of this paper is to show that under the following extra 
assumption of \cite{MuPa}:
\begin{equation}\label{W-nonphys2}
\forall F\in\mathbb{R}^{3\times 3} \quad |DW(F)| \leq C(|F|+1).
\end{equation}
 also the equilibria (possibly non-minimizing) 
of (\ref{elastic-En}) converge to the equilibria of (\ref{vonKarman})
or (\ref{I-intro}).
The definition of an equilibrium of the $3$d energy $J^h$  may be
understood in two different manners, corresponding to passing with 
the scaling $\epsilon$ of a variation $\phi$ to $0$ 
outside or inside the integral sign.  
Namely, for a fixed $h>0$, we may require that: 
\begin{equation}\label{equilib}
\forall \phi^h\in W^{1,2}(S^h,\mathbb{R}^3) \qquad \lim_{\epsilon\to 0} 
\frac{1}{\epsilon} 
\left(J^h(u^h + \epsilon \phi^h) - J^h(u^h)\right) = 0.
\end{equation}
or that:
\begin{equation}\label{equilib_form}
\forall \phi^h\in W^{1,2}(S^h,\mathbb{R}^3) \qquad 
\int_{S^h} DW(\nabla u^h):\nabla \phi^h = \int_{S^h} f^h \phi^h.
\end{equation}
The last condition is obtained by formal passing to the limit $\epsilon\to 0$
under the integral sign in (\ref{equilib}).
Integrating by parts  we also see that (\ref{equilib_form})
is the weak formulation of the following fundamental balance law \cite{Ball}:
\begin{equation}\label{EL3d} 
\mbox{div } [DW(\nabla u^h)] + f^h = 0 \mbox{ in } S^h, \qquad 
DW(\nabla u^h)\vec n = 0 \mbox{ on } \partial S^h, 
\end{equation}
where the operator $\mbox{div}$ above is understood as acting 
on rows of the matrix field $DW(\nabla u^h)$.
Whether the two definitions of equilibria (\ref{equilib}) and
(\ref{equilib_form}) are equivalent, even for local minimizers 
(without assuming extra regularity, e.g. their Lipschitz continuity) 
is an open problem of nonlinear elasticity, listed by Ball as Problem 5
in \cite{Ball}.

It turns out that the main convergence result described below follows with 
either  (\ref{equilib}) or (\ref{equilib_form}).  
The reason is that the difference between these two definitions 
(after an appropriate scaling), converges to $0$ with $h$, 
along particular sequences of variations $\phi^h$, 
which are however exactly the $3$d variations recovered from the
the variations of the $2$d functional $J^{vK}$ or $J$.

\begin{theorem}\label{th-main}
Assume (\ref{W-assump}) and (\ref{W-nonphys2}). Let $u^h\in W^{1,2}(S^h,
\mathbb{R}^3)$ be a sequence of deformations, satisfying:
\begin{itemize}
\item[(a)] the equilibrium equations (\ref{equilib_form}) hold,
\item[(b)] $E^h(u^h)\leq C e^h$, where $e^h$ is the scaling with
(\ref{scaling-intro}). 
\end{itemize}
Then there exist a sequence $Q^h\in SO(3)$, converging (up to a subsequence)
to some $\bar Q\in SO(3)$, and $c^h\in\mathbb{R}^3$ 
such that for the normalized rescaled deformations:
$$y^h(x+t\vec{n}) = (Q^h)^T u^h(x+h/h_0 t\vec{n}) - c^h$$ 
defined on the common domain $S^{h_0}$, we have:
\begin{enumerate}
\item[(i)] $y^h$ converge in $W^{1,2}(S^{h_0})$ to $\pi$.
\item[(ii)]  The scaled average displacements:
\begin{equation}\label{Vh-intro}
V^h(x) = \frac{h}{\sqrt{e^h}} \fint_{-h_0/2}^{h_0/2}
y^h(x+t\vec{n}) - x ~\mathrm{d}t
\end{equation} 
converge (up to a subsequence) in $W^{1,2}(S)$ to some $V\in \mathcal{V}$.
\item[(iii)] ${h}/{\sqrt{e^h}} ~\mathrm{sym }~\nabla V^h$ converge
(up to a subsequence) in $L^{2}(S)$ to some $B_{tan}\in\mathcal{B}$.
\item[(iv)] The triple $(V,B_{tan}, \bar Q)$ satisfies the Euler-Lagrange
equations of the functional $J^{vK}$. That is, 
for all $\tilde V\in \mathcal{V}$ with
$\tilde A = \nabla \tilde V$ given as in the formula (\ref{Adef-intro}), 
and all $\tilde B_{tan}\in\mathcal{B}$, 
there holds:
\begin{equation}\label{EL1}
\int_S \mathcal{L}_2\left(x, B_{tan} -\frac{\kappa}{2}(A^2)_{tan}\right)
: \tilde B_{tan} = 0,
\end{equation}
\begin{equation}\label{EL2} 
\begin{split}
&-\kappa \int_S \mathcal{L}_2\left(x, B_{tan}-\frac{\kappa}{2}
  (A^2)_{tan}\right) : (A\tilde A)_{tan}\\
& \qquad\qquad\qquad 
+ \frac{1}{12} \int_S \mathcal{L}_2\left(x, (\nabla(A\vec n) - A\Pi)_{tan}\right)
: (\nabla(\tilde A\vec n) - \tilde A\Pi)_{tan}
= \int_S f \cdot \bar Q \tilde V,
\end{split}
\end{equation}
When $\kappa=0$ then the couple $(V,\bar Q)$ satisfies (\ref{EL2}) 
for all $\tilde V\in\mathcal{V}$, which is the Euler-Lagrange equations of 
the functional (\ref{I-intro}).
\end{enumerate}
\end{theorem}

\begin{theorem}\label{th-main2}
Theorem \ref{th-main} remains true if in the assumption (a) the formal equilibrium
equation (\ref{equilib_form}) are replaced by the critical point condition
(\ref{equilib}).
\end{theorem}

We prove Theorem \ref{th-main} in section \ref{mainproof}
and Theorem \ref{th-main2} in section \ref{sec_observation}.
In section \ref{sec_3rdEL} we derive the third Euler-Lagrange equation
(after the first two (\ref{EL1}) and (\ref{EL2})), corresponding
to variation in $\bar Q\in SO(3)$.
We first notice that the limiting $\bar Q$ necessarily satisfies the
constraint of the average torque
$\tau(\bar Q) = \int_S f\times \bar Qx ~\mbox{d}x$ being $0$.
The main difficulty arises now from the fact that the variations must be taken 
inside $SO(3)$ in a way that this constraint remains satisfied.
Assuming that such variations exist, we establish the limit equation under the 
nondegeneracy condition that $Q^h$ approach $\bar Q$ along a direction 
$U\in T_{\bar Q}SO(3)$ for which $\partial_U \tau(\bar Q)\neq 0$.

\begin{remark}
Condition (\ref{W-nonphys2}) of \cite{MuPa} 
is of technical importance.
Notice that, in view of (\ref{W-assump}) resulting in 
$DW(F)=0$ for all $F\in SO(3)$, (\ref{W-nonphys2}) 
is equivalent to:
\begin{equation*}
\forall F\in\mathbb{R}^{3\times 3} \quad |DW(F)| \leq C \mbox{dist}(F, SO(3)).
\end{equation*}
Using the last assumption in (\ref{W-assump}), the above implies that:
$|DW(F)| \leq C W(F)^{1/2}$ for all $F\in\mathbb{R}^{3\times 3}$.
Hence,  roughly speaking, $W$ has a quadratic growth
and we see that (\ref{W-nonphys2}) is actually quite restrictive.
Independent from our research, Mora and Scardia \cite{MoSc} 
has presently established a result complementary
to ours where the requirement (\ref{W-nonphys2}) is relaxed, 
while the equilibrium condition of (\ref{total-intro}) is understood
in a different manner, related to Ball's inner variations and the 
Cauchy stress balance law \cite{Ball}. 
\end{remark}

\bigskip

\noindent{\bf Acknowledgments.} 
This work was partially supported by the NSF grant DMS-0707275 
and by the Center for Nonlinear Analysis (CNA) under
the NSF grants 0405343 and 0635983.

\section{Convergence of weak solutions to the 
  Euler-Lagrange equations (equilibria) of the $3$d energies}
\label{mainproof}

We first gather the relevant information from \cite{lemopa1}:

\begin{lemma}\label{approx}\cite{lemopa1}
Let $u^h\in W^{1,2}(S^h,\mathbb{R}^3)$ be a sequence of deformations 
of shells $S^h$.  Assume (\ref{scaling-intro}) 
and let the scaled energies $E^h(u^h)/e^h$ be uniformly bounded.
Then there exists a sequence of matrix fields $R^h\in W^{1,2}(S,\mathbb{R}^3)$
with $R^h(x)\in SO(3)$ for a.e. $x\in S$, and there exists a sequence of
matrices $Q^h\in SO(3)$ such that:
\begin{enumerate}
\item[(i)] $\|(Q^h)^T R^h -\mathrm{Id}\|_{W^{1,2}(S)}\leq C \sqrt{e^h}/h$.
\item[(ii)] $h/\sqrt{e^h}((Q^h)^T R^h - \mathrm{Id})$ converges (up to a
  subsequence) to a skew-symmetric matrix field $A$, weakly in $W^{1,2}(S)$.
\end{enumerate}
Moreover, there exists a sequence $c^h\in\mathbb{R}^3$ 
such that for the normalized rescaled deformations:
$$y^h(x+t\vec{n}) = (Q^h)^T u^h(x+h/h_0 t\vec{n}) - c^h$$ 
defined on the common domain $S^{h_0}$, the following holds.
\begin{enumerate}
\item[(iii)] $y^h$ converge in $W^{1,2}(S^{h_0})$ to $\pi$.
\item[(iv)]  The scaled average displacements $V^h$, defined in (\ref{Vh-intro})
converge (up to a subsequence) in $W^{1,2}(S)$ to some $V\in \mathcal{V}$,
whose gradient is given by $A$, as in (\ref{Adef-intro}).
\item[(v)] ${h}/{\sqrt{e^h}} ~\mathrm{sym }~\nabla V^h$ converge
(up to a subsequence) in $L^{2}(S)$ to some $B_{tan}\in\mathcal{B}$.
\end{enumerate}
\end{lemma}

The statements in Theorem \ref{th-main} (i), (ii), (iii) are
contained in the Lemma above. It therefore suffices to use the extra assumptions
(\ref{equilib_form}) and (\ref{W-nonphys2}) to recover equations 
(\ref{EL1}) and (\ref{EL2}) as $h\to 0$.

\medskip

We start by rewriting the equilibrium equation (\ref{equilib_form}) 
in a more convenient form.
Clearly, every variation $\phi^h\in W^{1,2}(S^h,\mathbb{R}^3)$ can be by a change 
of variables expressed as:
\begin{equation}\label{changevar}
\phi^h(x+t\vec n) = \psi(x+th_0/h \vec n),
\end{equation} 
for the corresponding $\psi\in W^{1,2}(S^{h_0}, \mathbb{R}^3)$. 
Then, (\ref{equilib_form}) becomes:
\begin{equation}\label{equilib2}
\begin{split}
&h^2\sqrt{e^h} \int_S f(x) \fint_{-h_0/2}^{h_0/2} \psi(x+t\vec n)
~\mbox{d}t~\mbox{d}x \\
&\qquad
 = h\int_S\fint_{-h_0/2}^{h_0/2} \mbox{det}(\mbox{Id} + th/h_0\Pi) 
DW(\nabla u^h(x+th/h_0\vec n)) : \nabla \phi^h(x+th/h_0\vec n)
~\mbox{d}t~\mbox{d}x. 
\end{split}
\end{equation}
Notice also that:
\begin{equation}\label{gradient}
\nabla \phi^h(x+th/h_0\vec n) = \nabla \psi(x+ t\vec n)\cdot P(x+t\vec n),
\end{equation}
where the matrix field $P\in L^\infty (S^{h_0},\mathbb{R}^3)$ has the
following non-zero entries:
$$ P(x+t\vec n)_{tan} = (\mbox{Id} + th/h_0\Pi(x))^{-1}(\mbox{Id} + t\Pi(x)),
\qquad \vec n^T P(x+t\vec n)\vec n = {h_0}/{h}.$$
In view of Lemma \ref{approx}, define the matrix fields
$E^h, G^h\in L^2(S^{h_0}, \mathbb{R}^{3\times 3})$:
$$ E^h = \frac{1}{\sqrt{e^h}} DW(\mbox{Id} + \sqrt{e^h} G^h),
\qquad G^h(x+t\vec n) = \frac{1}{\sqrt{e^h}} \left( (R^h)^T\nabla u^h(x+th/h_0
\vec n) - \mbox{Id}\right).$$
With this notation, recalling the frame invariance of $W$ in (\ref{W-assump})
we get, for every $F\in\mathbb{R}^{3\times 3}$:
\begin{equation*}
\begin{split}
\frac{1}{\sqrt{e^h}} DW(\nabla u^h(x+th/h_0\vec n)):F 
&= \frac{1}{\sqrt{e^h}} DW(R^h (\mbox{Id} + \sqrt{e^h}G^h)):F \\
& = \frac{1}{\sqrt{e^h}} DW(\mbox{Id} + \sqrt{e^h}G^h):(R^h)F
= R^h E^h : F.
\end{split}
\end{equation*}
In particular, (\ref{equilib2}) becomes, after exchanging $\psi$ to $(Q^h)^T\psi$, 
using (\ref{gradient}) and dividing both sides by $\sqrt{e^h}$:
\begin{equation}\label{equilib6}
\begin{split}
& h^2  \int_S f(x) \fint_{-h_0/2}^{h_0/2} Q^h\psi~\mbox{d}t\mbox{d}x \\
& = h \int_S\fint_{-h_0/2}^{h_0/2} \mbox{det}(\mbox{Id} + th/h_0\Pi) ~
\left[(Q^h)^T R^h(x) E^h(x+t\vec n)\right] : \nabla \phi^h(x+th/h_0\vec n)
~\mbox{d}t\mbox{d}x\\
& = h \int_S \fint_{-h_0/2}^{h_0/2} \mbox{det}(\mbox{Id} + th/h_0\Pi) ~
\left[(Q^h)^T R^hE^h\right]_{TS} : \left[(\nabla_{tan} \psi)(\mbox{Id} +th/h_0\Pi)^{-1}
(\mbox{Id} +t\Pi)\right] ~\mbox{d}t\mbox{d}x\\
&\qquad
+ h_0 \int_S \fint_{-h_0/2}^{h_0/2} \mbox{det}(\mbox{Id} + th/h_0\Pi) ~
((Q^h)^T R^h E^h \vec n) ~\partial_{\vec n} \psi(x+t\vec n)~\mbox{d}t\mbox{d}x,
\end{split}
\end{equation}
where $\nabla_{tan}$ denotes gradient in the tangent directions of $T_xS$. 
The subscript $TS$ stands for taking the $3\times 2$ minor of the matrix under
consideration,  for example: $\nabla_{tan}\psi = [\nabla\psi]_{TS}$.

\begin{lemma}\label{conv2}
The sequence $G^h$ converges (up to a subsequence), weakly in $L^2(S^{h_0},
\mathbb{R}^{3\times 3})$ to an $L^2(S^{h_0})$ matrix field $G$, whose
tangential minor has the form:
\begin{equation}\label{Gtan}
G(x+t\vec n)_{tan} = \left(B_{tan}-\frac{\kappa}{2} (A^2)_{tan}\right)
+ \frac{t}{h_0}  \left(\nabla(A\vec n) - A\Pi \right)_{tan}.
\end{equation} 
Moreover, if (\ref{W-nonphys2}) holds, then:
\begin{enumerate}
\item[(i)] $E^h$ converges (up to a subsequence) weakly in 
$L^2(S^{h_0},\mathbb{R}^{3\times 3})$ to the matrix field 
$E= \mathcal{L}_3 G$. 
\item[(ii)]The sequence $(Q^h)^T R^h (x) E^h(x+t\vec n)$ converges (up to a
  subsequence) to $E$, weakly in  $L^2(S^{h_0},\mathbb{R}^{3\times 3})$.
\end{enumerate}
\end{lemma}
\begin{proof}
The convergence of $G^h$ and the formula (\ref{Gtan}) follow from 
Lemma 3.6 and Lemma 4.1 in \cite{lemopa1}. 
Convergence in (i) is a consequence of Proposition 2.3 in \cite{MuPa}, where
the crucial role was played by the following equivalent form of the assumption
(\ref{W-nonphys2}):
$$ \forall F\in \mathbb{R}^{3\times 3} \qquad
|DW(\mbox{Id} + F)|\leq C |F|.$$
Finally, (iii) is an immediate consequence of (ii) in view of Lemma
\ref{approx} (i) and the boundedness of $(Q^h)^T R^h$ in $L^\infty(S^{h_0})$.
\end{proof}

\begin{lemma}\label{Eprop}
The matrix field $E\in L^2(S^{h_0}, \mathbb{R}^{3\times 3})$, defined in Lemma
\ref{conv2} (i) satisfies the following properties, a.e. in $S^{h_0}$:
\begin{enumerate}
\item[(i)] $E\vec n = 0$.
\item[(ii)] $E^T = E$, that is: $E$ is symmetric.  
\item[(iii)] $E_{tan}(x+t\vec n) = \mathcal{L}_2(x, G_{tan}(x+t\vec n))$.
\end{enumerate}
\end{lemma}
\begin{proof}
To prove (i), one needs to pass $h\to 0$ in (\ref{equilib6}) and use 
Lemma \ref{conv2} (ii) to obtain:
\begin{equation}\label{uno}
\int_S \fint_{-h_0/2}^{h_0/2} \Big(E(x+t\vec n)\vec
  n\Big)~ \partial_{\vec n} \psi(x+t\vec n)~\mbox{d}t\mbox{d}x=0.
\end{equation}
Now, any vector field $\phi\in L^2(S^{h_0}, \mathbb{R}^3)$ has the form
$\phi = \partial_{\vec n}\psi$, where $\psi(x+t\vec n)= \int_{-h_0/2}^t
\phi(x+s\vec n)~\mbox{d}s$.  Therefore (i) follows from (\ref{uno}).

By frame indifference (\ref{W-assump}) and the fact that $W$ is minimized at 
$\mbox{Id}$, it follows that $DW(F)=0$ for all $F\in SO(3)$.  It implies that
for all $H\in so(3)$ there holds $\mathcal{L}_3H=0$, and so: 
$E:H = \mathcal{L}_3G:H = \mathcal{L}_3H:G = 0$, proving (ii).

The assertion (iii) follows from $E=\mathcal{L}_3G$ and the reasoning exactly 
as in the proof of Proposition 3.2 \cite{MuPa}.
\end{proof}

A more precise information, with respect to that in Lemma \ref{Eprop} (ii)
is given by:

\begin{lemma}\label{Ehsym}
There holds: 
\begin{enumerate}
\item[(i)] $\|~\mathrm{ skew }~E^h\|_{L^1(S^{h_0})} \leq C\sqrt{e^h}$.
\item[(ii)] $\displaystyle{\lim_{h\to 0} \frac{1}{h} 
\|~\mathrm{ skew }~E^h\|_{L^p(S^{h_0})} =0,}$ for some exponent $p\in (1,2)$.
\end{enumerate}
\end{lemma}
\begin{proof}
By frame indifference (\ref{W-assump}) one has:
$0 = DW(F):HF = DW(F)F^T:H$, for all $F\in\mathbb{R}^{3\times 3}$ and all
$H\in so(3)$ (since $HF$ is a tangent vector to $SO(3)F$ at $F$). 
We further obtain that $DW(F) F^T$ is a symmetric matrix. Apply this statement 
pointwise to the matrix field $F=\mbox{Id} + \sqrt{e^h} G^h$:
\begin{equation*}
\begin{split}
0 &= \frac{1}{\sqrt{e^h}} \Big(DW(\mbox{Id} + \sqrt{e^h} G^h)  
~(\mbox{Id} + \sqrt{e^h} (G^h)^T) - (\mbox{Id} + \sqrt{e^h} G^h)~
DW^T(\mbox{Id} + \sqrt{e^h} G^h)\Big)\\
&= E^h - (E^h)^T + \sqrt{e^h}\left(E^h (G^h)^T - G^h(E^h)^T\right).
\end{split}
\end{equation*}
Hence the claim in (i) is proved, as by Lemma \ref{conv2}:
$$\|\mbox{sym~} (E^h (G^h)^T)\|_{L^1(S^{h_0})} 
\leq C \|E^h\|_{L^2(S^{h_0})} \|G^h\|_{L^2(S^{h_0})} \leq C.$$

\smallskip

Now, (ii) follows from (i) in view of the boundedness of $E^h$ in $L^2(S^{h_0})$,
(\ref{scaling-intro}), and through an interpolation inequality: 
\begin{equation*}
\frac{1}{h} \|~\mathrm{ skew }~E^h\|_{L^p(S^{h_0})} \leq
\frac{1}{h} \|~\mathrm{ skew }~E^h\|_{L^1}^\theta 
 \|~\mathrm{ skew }~E^h\|_{L^2}^{1-\theta} \leq
{C}/{h}\sqrt{e^h}^\theta = C \left(\sqrt{e^h}/{h^2}\right)^\theta
h^{2\theta - 1}, 
\end{equation*}
where $1/p = \theta + (1-\theta)/2$ and $\theta\in (0,1)$.
Clearly, the above converges to $0$, when $\theta>1/2$.
\end{proof}

Introduce now the two matrix fields $\bar E, ~\hat E\in L^2(S,\mathbb{R}^3)$
given by the 0th and 1st moments of $E$:
\begin{equation*}
\bar E(x) = \fint_{-h_0/2}^{h_0/2} E(x+t\vec n)~\mbox{d}t,
\qquad \hat E(x) = \fint_{-h_0/2}^{h_0/2} tE(x+t\vec n)~\mbox{d}t.
\end{equation*}
It easily follows by Lemma \ref{Eprop} (iii) and Lemma \ref{conv2} that:
\begin{equation}\label{Ebar}
\bar E_{tan}(x) 
= \fint_{-h_0/2}^{h_0/2} \mathcal{L}_2(x, G_{tan}(x+t\vec n))~\mbox{d}t
= \mathcal{L}_2\left(x, B_{tan}-\frac{\kappa}{2} (A^2)_{tan}\right),
\end{equation}
\begin{equation}\label{Ehat}
\hat E_{tan}(x) 
= \fint_{-h_0/2}^{h_0/2} \mathcal{L}_2(x, tG_{tan}(x+t\vec n))~\mbox{d}t
= \frac{h_0}{12} \mathcal{L}_2\left(x, (\nabla(A\vec n) - A\Pi)_{tan}\right).
\end{equation}
We will now use the fundamental balance (\ref{equilib6}) and the above formulas
to recover the Euler-Lagrange equations (\ref{EL1}), (\ref{EL2}) in the limit
as $h\to 0$.

\bigskip

\noindent {\bf Proof of the first Euler-Lagrange equation (\ref{EL1}).}

\noindent Use the variation of the form: $\psi(x+t\vec n) = \phi(x)$ 
in (\ref{equilib6}), divide both sides by $h$ and pass to the limit to obtain:
\begin{equation}\label{zero}
\begin{split}
0 &= \lim_{h\to 0} \int_S\fint_{-h_0/2}^{h_0/2} \mbox{det}(\mbox{Id} + th/h_0\Pi) ~
\left[(Q^h)^T R^h E^h\right]_{TS} : \left[\nabla_{tan} \phi(x)
(\mbox{Id} +th/h_0\Pi)^{-1}\right] ~\mbox{d}t\mbox{d}x\\
&= \int_S\fint_{-h_0/2}^{h_0/2} 
E_{TS}: \nabla_{tan} \psi(x)~\mbox{d}t\mbox{d}x
= \int_S\mathcal{L}_2\left(x, B_{tan}-\frac{\kappa}{2} (A^2)_{tan}\right) :
[\nabla \phi(x)]_{tan}~\mbox{d}x\\
&= \int_S\mathcal{L}_2\left(x, B_{tan}-\frac{\kappa}{2} (A^2)_{tan}\right) :
\mbox{sym }\nabla \phi(x)~\mbox{d}x
\end{split}
\end{equation}
where we have used Lemma \ref{conv2} (i), Lemma \ref{Eprop} and (\ref{Ebar}).
Therefore, by density of $\{\mbox{sym }\nabla \phi\}$ in the space $\mathcal{B}$, 
(\ref{EL1}) follows immediately.

\bigskip

\noindent {\bf Proof of the second Euler-Lagrange equation (\ref{EL2}).}

\noindent  Let $\tilde V\in \mathcal{V}$ and denote by $\tilde A$ the
skew-symmetric matrix field representing $\nabla \tilde V$, as in 
(\ref{Adef-intro}).

\smallskip

{\bf 1.} We now apply (\ref{equilib6}) to a variation of the form:
$\psi(x+t\vec n) = t\tilde A\vec n(x)$. For simplicity, write
$\eta = \tilde A\vec n\in W^{1,2}(S,\mathbb{R}^3)$.
Upon dividing (\ref{equilib6}) by $h$ and passing to the limit, we obtain:
\begin{equation}\label{try}
\begin{split}
0 &= \lim_{h\to 0} \Bigg[\int_S\fint_{-h_0/2}^{h_0/2} \mbox{det}(\mbox{Id} +
th/h_0\Pi) ~\left[(Q^h)^T R^h tE^h\right]_{TS} : \left[\nabla_{tan} \eta(x)
(\mbox{Id} +th/h_0\Pi)^{-1}\right] ~\mbox{d}t\mbox{d}x\\
& \qquad \qquad
+ \frac{h_0}{h} \int_S\fint_{-h_0/2}^{h_0/2}((Q^h)^T R^h E^h\vec n)~ \eta(x)
~\mbox{d}t\mbox{d}x\\
& \qquad \qquad
+ \int_S\fint_{-h_0/2}^{h_0/2}(t\mbox{ trace }\Pi + t^2h/h_0 \det \Pi)
((Q^h)^T R^h E^h\vec n)~ \eta(x)~\mbox{d}t\mbox{d}x\Bigg],
\end{split}
\end{equation}
where we used the identity:
$$\det (\mbox{Id} + th/h_0\Pi) = 1 + th/h_0 \mbox{trace }\Pi + t^2 h^2/h_0^2
\det \Pi. $$
The first term in (\ref{try}), in view of Lemma \ref{conv2} (ii),
Lemma \ref{Eprop} and (\ref{Ehat}), converges to:
\begin{equation*}
\begin{split}
\int_S\fint_{-h_0/2}^{h_0/2} tE_{TS}: \nabla_{tan} \eta(x)~\mbox{d}t\mbox{d}x
&= \int_S\hat E_{tan} : (\nabla \eta(x))_{tan}~\mbox{d}t\mbox{d}x \\
&= \frac{h_0}{12} \int_S \mathcal{L}_2\left(x, (\nabla(A\vec n) -
  A\Pi)_{tan}\right) : (\nabla \eta(x))_{tan} ~\mbox{d}x.
\end{split}
\end{equation*}
In turn, the third term in (\ref{try}) converges to $0$.  
This is because $(Q^h)^T R^h E^h\vec n$ converge weakly in
$L^2(S^{h_0},\mathbb{R}^3)$ to $E\vec n = 0$, by Lemma \ref{conv2} (ii) 
and Lemma \ref{Eprop} (i). Summarizing, (\ref{try}) yields:
\begin{equation}\label{due}
\begin{split} 
\lim_{h\to 0} \frac{1}{h} 
\int_S\fint_{-h_0/2}^{h_0/2}((Q^h)^T R^h E^h\vec n)~ \tilde A\vec n~\mbox{d}t\mbox{d}x
= -\frac{1}{12} \int_S \mathcal{L}_2\left(x, (\nabla(A\vec n) -
  A\Pi)_{tan}\right) : (\nabla (\tilde A\vec n))_{tan} ~\mbox{d}x.
\end{split}
\end{equation}

\smallskip 

{\bf 2.} Now,  apply (\ref{equilib6}) to the variation
$\psi(x+t\vec n) = \tilde V(x)$, and pass to the limit after dividing 
both sides of (\ref{equilib6}) by $h^2$:
\begin{equation}\label{tre}
\begin{split}
\int_S f(x)\cdot& \bar{Q}\tilde V(x)~\mbox{d}x 
 = \lim_{h\to 0}\int_S f(x)\cdot Q^h\tilde V(x)~\mbox{d}x \\
& = \lim_{h\to 0} \Bigg[ 
\int_S\fint_{-h_0/2}^{h_0/2}
\left[ \frac{1}{h}((Q^h)^T R^h - \mbox{Id}) E^h\right]_{TS} 
: \left[\tilde A(x)_{TS} (\mbox{Id} + th/h_0\mbox{adj }\Pi)\right]
~\mbox{d}t\mbox{d}x\\
& \qquad\qquad
+ \int_S\fint_{-h_0/2}^{h_0/2}\frac{1}{h} E^h_{TS}: 
\left[\tilde A(x)_{TS} (\mbox{Id} + th/h_0\mbox{adj }\Pi)\right]
~\mbox{d}t\mbox{d}x\Bigg]\\
& :=  \lim_{h\to 0} [I_h + II_h],
\end{split}
\end{equation}
where we used the definition of the adjoint matrix:
$$\mbox{det}(\mbox{Id} + th/h_0\Pi) ~(\mbox{Id} +
th/h_0\Pi)^{-1} = \mbox{adj } (\mbox{Id} + th/h_0\Pi) 
=  \mbox{Id} + th/h_0 \mbox{adj } \Pi. $$
Notice that, by Lemma \ref{approx} (ii) and 
(\ref{scaling-intro}), the matrix field:
$$ 1/h((Q^h)^T R^h - \mbox{Id}) = (\sqrt{e^h}/h^2)
h/\sqrt{e^h} ((Q^h)^T R^h - \mbox{Id})$$ 
converges to $\kappa A$, weakly in $W^{1,2}(S)$ and hence strongly in $L^2(S)$. 
Hence, by the weak convergence of $E^h$ to
$E$ and the uniform convergence of $ (\mbox{Id} + th/h_0\mbox{adj }\Pi)$ to
$\mbox{Id}$, the first term of (\ref{tre}) converges to:
\begin{equation}\label{Ih}
\begin{split}
\lim_{h\to 0} I_h & = \kappa \int_S\fint_{-h_0/2}^{h_0/2}
(AE)_{TS} : \tilde A_{TS}
= \kappa \int_S (A\bar E)_{TS}: \tilde A_{TS}
= \kappa \int_S (A\bar E): \tilde A \\
& = - \kappa \int_S \bar E: (A\tilde A)
=  - \kappa \int_S \bar E_{tan}: (A\tilde A)_{tan}\\
& = - \kappa \int_S \mathcal{L}_2\left(x, B_{tan}-\frac{\kappa}{2}
  (A^2)_{tan}\right) : (A\tilde A)_{tan} ~\mbox{d}x,
\end{split}
\end{equation}
where we also have used Lemma \ref{Eprop} and (\ref{Ebar}).

\smallskip 

{\bf 3.} Towards finding the limit of $II_h$ in (\ref{tre}),
consider first the contribution of the tangential minors. 
By Lemma \ref{Ehsym} (ii) and since $\tilde A\in L^p(S^{h_0})$ for
all $p\geq 1$, one observes that:
\begin{equation}\label{help}
\lim_{h\to 0} \frac{1}{h}\int_S \fint_{-h_0/2}^{h_0/2} \mbox{ skew } E^h_{tan} :
\tilde A_{tan} = 0.
\end{equation}
Hence:
\begin{equation}\label{blah1}
\begin{split}
\lim_{h\to 0} \int_S&\fint_{-h_0/2}^{h_0/2} \frac{1}{h} E^h_{tan} :
\left[\tilde A(x)_{tan}(\mbox{Id} + th/h_0 \mbox{ adj }\Pi)\right]
~\mbox{d}t\mbox{d}x\\
& = \frac{1}{h_0}\lim_{h\to 0} \int_S\fint_{-h_0/2}^{h_0/2} tE^h_{tan} :
\left[\tilde A_{tan} \mbox{ adj } \Pi\right]
= \frac{1}{h_0}\lim_{h\to 0} \int_S \hat E_{tan} : 
\left[\tilde A_{tan} \mbox{ adj } \Pi\right]\\
& = - \frac{1}{h_0}\lim_{h\to 0} \int_S \hat E_{tan} : (\tilde A_{tan} \Pi)^T
= - \frac{1}{12} \int_S \mathcal{L}_2\left(x, (\nabla(A\vec n) 
- A\Pi)_{tan}\right) : (\tilde A\Pi)_{tan}~\mbox{d}x,
\end{split}
\end{equation}
where we have used (\ref{Ehat}) and Lemma \ref{Eprop} (ii), combined with
the following formula, which can be easily checked for $\tilde A_{tan}\in so(2)$:
$$ \tilde A_{tan} \mbox{ adj } \Pi = -(\tilde A_{tan}\Pi)^T.$$

\smallskip

Further, by (\ref{due}):
\begin{equation}\label{blah2}
\begin{split}
\lim_{h\to 0} \int_S\fint_{-h_0/2}^{h_0/2} \frac{1}{h} &\Big((E^h)^T\vec n\Big) 
\Big((\tilde A)^T\vec n\Big)~\mbox{d}t\mbox{d}x\\
&= - \lim_{h\to 0} \frac{1}{h} \int_S\fint_{-h_0/2}^{h_0/2} ((Q^h)^T R^h E^h \vec n)
(\tilde A \vec n) \\
& \qquad \qquad 
+ \lim_{h\to 0} \int_S\fint_{-h_0/2}^{h_0/2} \left[\frac{1}{h} ((Q^h)^TR^h -\mbox{Id})
(E^h\vec n)\right] (\tilde A\vec n)\\
& \qquad \qquad 
+ 2 \lim_{h\to 0} \int_S\fint_{-h_0/2}^{h_0/2} \left[\frac{1}{h}
(\mbox{ skew } E^h)\vec n\right] (\tilde A\vec n)\\
& = \frac{1}{12} \int_S \mathcal{L}_2\left(x, (\nabla(A\vec n) 
- A\Pi)_{tan}\right) : (\nabla(\tilde A\vec n))_{tan}~\mbox{d}x.
\end{split}
\end{equation}
Indeed, $1/h ((Q^h)^TR^h - \mbox{Id})$ converges to $\kappa A$
weakly in $L^4(S)$ while $\tilde A\vec n\in L^4(S)$ and $\bar E^h\vec n$
converges to $0$ weakly in $L^2(S)$.  Therefore the second term in (\ref{blah2}) 
converges to $0$.  The last limiting term there vanishes as well, by Lemma
\ref{Ehsym} (ii) as in (\ref{help}).

\smallskip

Finally, we have:
\begin{equation}\label{blah3}
\lim_{h\to 0} \frac{1}{h_0}\int_S\fint_{-h_0/2}^{h_0/2} \Big(\vec n^T tE^h\Big)_{tan}
(\mbox{ adj }\Pi) \Big((\tilde A)^T \vec n\Big)_{tan} 
~\mbox{d}t\mbox{d}x = 0,
\end{equation}
because $(\hat E^h)^T\vec n$ converges to $0$ weakly in $L^2(S)$ by
Lemma \ref{conv2} (i) and Lemma \ref{Eprop}.

\smallskip

Adding now (\ref{blah1}), (\ref{blah2}) and (\ref{blah3}) we obtain:
\begin{equation}\label{IIh}
\lim_{h\to 0} II_h = \frac{1}{12} \int_S \mathcal{L}_2\left(x, (\nabla(A\vec n) 
- A\Pi)_{tan}\right) : (\nabla(\tilde A\vec n) - \tilde A\Pi)_{tan}
~\mbox{d}x.
\end{equation}
Together with (\ref{tre}) and (\ref{Ih}), the formula 
(\ref{IIh}) implies (\ref{EL2}).
\endproof

\section{Convergence of critical points of the $3$d
energy functionals}\label{sec_observation}

In this section we prove Theorem \ref{th-main2}.  Proceeding 
as in the proof of Theorem  \ref{th-main}, one needs to exchange the expression  
$\int_{S^h} DW(\nabla u^h) : \nabla \phi^h$ by that of 
$\lim_{\epsilon\to 0} \int_{S^h} \frac{1}{\epsilon}
\left[W(\nabla u^h + \epsilon \nabla \phi^h) - W(\nabla u^h)
\right]~\mbox{d}z$.  
As shown below, the error given by the difference of these two quantities,
converges to $0$ as $h\to 0$, after 
an appropriate scaling by powers of $h$ and $\sqrt{e^h}$ and along the 
variations $\phi^h$ used in the proof of Theorem \ref{th-main}.

We first prove a more general lemma, in which we derive the optimal 
asymptotic properties a sequence $u^h$ must satisfy 
in order that the conclusions of Theorem \ref{th-main} hold true. 
These properties will later be established 
for the critical points (\ref{equilib}) of the functional $J^h$. 

\begin{lemma}\label{general}
Assume (\ref{W-assump}) and (\ref{W-nonphys2}). Let $u^h\in W^{1,2}(S^h,
\mathbb{R}^3)$ be a sequence of deformations, satisfying:
$E^h(u^h)\leq C e^h$ where the scaling $e^h$ is as in (\ref{scaling-intro}).
For every $\psi\in W^{1,2}(S^{h_0}, \mathbb{R}^3)$, consider the rescaled
variations $\phi^h$ given by (\ref{changevar}) and define the corresponding
error terms:
$$\mathcal{E}_h(\psi) = \int_{S^h} DW(\nabla u^h) : \nabla\phi^h - \int_{S^h}
f^h\phi^h.$$ 
Then all assertions of Theorem \ref{th-main} hold, provided that:
\begin{enumerate}
\item[(i)] $\displaystyle{\lim_{h\to 0} \frac{1}{\sqrt{e^h}} 
\mathcal{E}_h(\psi) = 0,}$ for all  $\psi\in W^{1,2}(S^{h_0}, \mathbb{R}^3)$,
\item[(ii)] $\displaystyle{\lim_{h\to 0} \frac{1}{h\sqrt{e^h}}
    \mathcal{E}_h(\psi) = 0,}$
for all  $\psi$ of the form $\psi(x+t\vec n) = \phi(x)$, $\phi\in W^{1,2}(S,
\mathbb{R}^3)$, and all  $\psi$ of the form $\psi(x+t\vec n) = t\tilde A\vec
n(x)$ with $\tilde A$ given as in (\ref{Adef-intro}) for some 
$\tilde V\in \mathcal{V}$,
\item[(iii)] $\displaystyle{\lim_{h\to 0} \frac{1}{h^2\sqrt{e^h}} 
\mathcal{E}_h(\psi) = 0,}$
for all  $\psi$  of the form $\psi(x+t\vec n) = \tilde V(x)$ with 
$\tilde V\in \mathcal{V}$.
\end{enumerate}
\end{lemma}
\begin{proof}
The proof follows by a direct inspection of the proof of Theorem
\ref{th-main}.  Indeed, (i) is needed to derive (\ref{uno}),
(ii) serves for getting (\ref{zero}) and (\ref{due}), through (\ref{try}),
while (iii) implies (\ref{tre}).
\end{proof}

Theorem \ref{th-main2} is now a consequence of the following
observation:
\begin{lemma}
Assume (\ref{W-assump}) and (\ref{W-nonphys2}). Let $u^h\in W^{1,2}(S^h,
\mathbb{R}^3)$  satisfy:
$E^h(u^h)\leq C e^h$ with the scaling $e^h$ is as in (\ref{scaling-intro}).
If (\ref{equilib}) holds then the conditions (i), (ii), (iii)
in Lemma \ref{general} are fulfilled.
\end{lemma}
\begin{proof}
In view of (\ref{equilib}), it is enough to prove that (i), (ii), and (iii)
in Lemma \ref{general} hold with $\mathcal{E}_h(\psi)$ replaced by 
a more convenient error quantity:
$$\mathcal{E}'_h(\psi) = \lim_{\epsilon\to 0} \mathcal{E}'_{h,\epsilon}(\psi),
\qquad
\mathcal{E}'_{h,\epsilon}(\psi) = 
\int_{S^h} \left[\frac{1}{\epsilon} \left(W(\nabla u^h + \epsilon\nabla\phi^h)
- W(\nabla u^h)\right)\right] - DW(\nabla u^h) : \nabla\phi^h ~\mbox{d}z,$$
and the rescaled variations $\phi^h$ given by (\ref{changevar}).

\smallskip

Define the good sets: 
$\Omega_{h,\epsilon}= \{z\in S^h; ~\mbox{dist}(\nabla u^h(z), 
SO(3)) <\delta  \mbox{ and } \epsilon |\nabla\phi^h(z)|<\delta\} $,
with $\delta>0$ small enough for $W$ to be $\mathcal{C}^2$ in the open
neighborhood $\{F\in\mathbb{R}^{3\times 3}; ~ \mbox{dist}(F, SO(3))<3\delta\}.$
We will estimate $\mathcal{E}'_{h,\epsilon}(\psi)$ by writing it as a sum 
of two integrals: one over $\Omega_{h,\epsilon}$ and the other over
$S^h\setminus \Omega_{h,\epsilon}$. Apply the mean value theorem to 
the continuous function $DW$ in the first integral, while in the second integral 
we use the assumption (\ref{W-nonphys2}):
\begin{equation}\label{m1}
\begin{split}
&\left|\mathcal{E}'_{h,\epsilon}(\psi)\right|  
 = \left|\int_{S^h} \int_0^1 \left[DW(\nabla u^h + \epsilon s \nabla\phi^h)
- DW(\nabla u^h)\right]~\mbox{d}s : \nabla\phi^h~\mbox{d}z\right|\\
& \qquad
\leq C\epsilon \int_{\Omega_{h,\epsilon}}|\nabla\phi^h|~\mbox{d}z\\
& \qquad\qquad
 + C\int_{S^h\setminus \Omega_{h,\epsilon}}\int_0^1 
\left[\mbox{dist}(\nabla u^h + \epsilon s \nabla\phi^h, SO(3)) + 
\mbox{dist}(\nabla u^h, SO(3))\right]~\mbox{d}s 
\cdot |\nabla\phi^h|~\mbox{d}z\\
& \qquad
\leq  C\epsilon \int_{\Omega_{h,\epsilon}}|\nabla\phi^h|
+ C \int_{S^h\setminus \Omega_{h,\epsilon}} 
\mbox{dist}(\nabla u^h, SO(3)) |\nabla\phi^h|
+C\epsilon \int_{S^h\setminus \Omega_{h,\epsilon}}|\nabla\phi^h|^2.
\end{split}
\end{equation}
We see that the first and the third term above converge to $0$ as $\epsilon\to 0$.
To treat the second term, notice that by (\ref{W-assump}), the energy bound, 
and (\ref{gradient}):
\begin{equation}\label{m2}
\begin{split}
\int_{S^h\setminus \Omega_{h,\epsilon}}
\mbox{dist}&(\nabla u^h, SO(3))|\nabla\phi^h|
\leq C \int_{S^h\setminus \Omega_{h,\epsilon}} W(\nabla u^h)^{1/2} |\nabla\phi^h|\\
& \leq C \left[hE^h(u^h)\right]^{1/2} 
\|\nabla\phi^h\|_{L^2(S^h\setminus \Omega_{h,\epsilon})}   \leq
C h^{1/2}\sqrt{e^h} \|\nabla\phi^h\|_{L^2(S^h\setminus \Omega_{h,\epsilon})} 
\end{split}
\end{equation}
Observe also that:
$$\int_{S^h\setminus \Omega_{h,\epsilon}}|\nabla\phi^h|^2
\leq C \int_{S^{h_0}\setminus \omega_{h,\epsilon}}\left[h|\nabla_{tan}\psi|^2
+ \frac{1}{h}|\partial_{\vec n}\psi|^2\right]~\mbox{d}z,$$
where the set $\omega_{h,\epsilon}=S^{h_0}\setminus \left\{x+t\vec n; 
~ x+th/h_0\vec n\in \Omega_{h,\epsilon}\right\}$. Its measure can be
estimated as:
\begin{equation}\label{m3}
\begin{split}
|\omega_{h,\epsilon}|&\leq {C}/{h}|S^h\setminus \Omega_{h,\epsilon}|\\
&\leq {C}/{h} \left\{ |z\in S^h; ~ \mbox{dist}(\nabla u^h(z), SO(3))\geq \delta|
+ |z\in S^h; ~ \epsilon |\nabla\phi^h(z)|\geq\delta|\right\}\\
&\leq {C}/{h}\left\{\int_{S^h} W(\nabla u^h) 
+ \epsilon^2\int_{S^h}|\nabla\phi^h|^2\right\}.
\end{split}
\end{equation}
In particular:
\begin{equation}\label{m5}
\lim_{h\to 0}\lim_{\epsilon\to 0} |\omega_{h,\epsilon}| = 0.
\end{equation}

\smallskip

We now prove (i). Passing to the limit in (\ref{m1}) and (\ref{m2}), and using 
(\ref{m3}) with (\ref{m5}) we obtain:
\begin{equation*}
\begin{split}
\lim_{h\to 0}\lim_{\epsilon\to 0} \frac{1}{\sqrt{e^h}} 
|\mathcal{E}'_{h,\epsilon}(\psi)|& \leq C \lim_{h\to 0}\lim_{\epsilon\to 0}
h^{1/2} \|\nabla\phi^h\|_{L^2(S^h\setminus \Omega_{h,\epsilon})} \\
& \leq \lim_{h\to 0}\lim_{\epsilon\to 0}\left(h\|\nabla_{tan}\psi\|_{L^2(S^{h_0})}
+ \|\partial_{\vec n}\psi\|_{L^2(\omega_{h,\epsilon})} \right)= 0.
\end{split}
\end{equation*}

\smallskip

\noindent To prove (ii), consider:
\begin{equation*}
\begin{split}
\lim_{h\to 0}\lim_{\epsilon\to 0} \frac{1}{h\sqrt{e^h}} 
|\mathcal{E}'_{h,\epsilon}(\psi)| & \leq 
C \lim_{h\to 0}\lim_{\epsilon\to 0} h^{-1/2} 
\|\nabla\phi^h\|_{L^2(S^h\setminus \Omega_{h,\epsilon})}\\
& \leq C \lim_{h\to 0}\lim_{\epsilon\to 0}
\left(\|\nabla_{tan}\psi\|_{L^2(\omega_{h,\epsilon})}
+ \frac{1}{h} \|\partial_{\vec n}\psi\|_{L^2(\omega_{h,\epsilon})}\right).
\end{split}
\end{equation*}
The first limit above is $0$ by (\ref{m5}).
Concerning the second term, it may be dropped for $\psi(x+t\vec n) = \phi(x)$,
while in the other case when $\phi(x+t\vec n) = t\tilde A\vec n(x)$ we
have $\partial_{\vec n}\psi = \tilde A\vec n\in W^{1,2}(S)$ and hence:
\begin{equation*}
\begin{split}
\lim_{h\to 0}\lim_{\epsilon\to 0} \frac{1}{h} 
\|\partial_{\vec n}\psi\|_{L^2(\omega_{h,\epsilon})}
\leq \lim_{h\to 0}\lim_{\epsilon\to 0} {C}/{h} |\omega_{h,\epsilon}|^{1/3}
\|\partial_{\vec n}\psi\|_{L^6(S)} \leq \lim_{h\to 0} {C}/{h} (e^h)^{1/3}
= 0,
\end{split}
\end{equation*}
in view of (\ref{scaling-intro}).

\smallskip

To prove (iii) for $\psi(x+t\vec n) = \tilde V(x)$, recall that 
$\nabla \tilde V\in W^{1,2}(S,\mathbb{R}^3)$ and write:
\begin{equation*}
\begin{split}
\lim_{h\to 0}\lim_{\epsilon\to 0} \frac{1}{h^2\sqrt{e^h}} 
|\mathcal{E}'_{h,\epsilon}(\psi)| \leq 
\lim_{h\to 0}\lim_{\epsilon\to 0}  {C}/{h}
\|\nabla\tilde V\|_{L^2(\omega_{h,\epsilon})}
\leq \lim_{h\to 0}\lim_{\epsilon\to 0} {C}/{h} |\omega_{h,\epsilon}|^{1/3}
\|\nabla\tilde V\|_{L^6(S)} = 0,
\end{split}
\end{equation*}
as before. This achieves the proof.
\end{proof}

\section{The limiting rotations $\bar Q$} \label{sec_3rdEL}

In this section we will derive the third Euler-Lagrange equation
(after the first two (\ref{EL1}) and (\ref{EL2})), corresponding
to variation in $\bar Q\in SO(3)$, and under certain nondegeneracy condition.
We first notice that the limiting $\bar Q$ necessarily satisfies the
constraint of the average torque:
\begin{equation}\label{torque}
\tau(\bar Q) = \int_S f\times \bar Qx ~\mbox{d}x = 0. 
\end{equation}
The main difficulty arises now from the fact that the variations must be taken 
inside $SO(3)$ in a way that this constraint remains satisfied.
Assuming that such variations exist, we establish the limit equation under the 
additional condition that $Q^h$ approach $\bar Q$ along a direction 
$U\in T_{\bar Q}SO(3)$ for which $\partial_U \tau(\bar Q)\neq 0$.

In what follows, the crucial role is played by
the function $g(Q) = \int_S f\cdot Qx~\mbox{d}x$
defined on $SO(3)$. Let $B\in \mathbb{R}^{3\times 3}$ be such that: 
$g(Q) = B:Q$, for all $Q\in SO(3)$.

\begin{lemma}
Assume the hypothesis of Theorem \ref{th-main} or Theorem \ref{th-main2}. 
Then the limit
$\bar Q\in SO(3)$ of $Q^h$ must satisfy:
\begin{equation}\label{r0}
\int_S f\cdot \bar Q Fx~\mathrm{d}x = 0 \qquad \forall F\in so(3),
\end{equation}
or equivalently (\ref{torque}).
Another equivalent formulation of (\ref{r0}) is: $\mathrm{skew }(\bar Q^T B)=0$.
\end{lemma}
\begin{proof}
First, for any given $H\in so(3)$, consider the variation $\phi^h = Hu^h$
in the equilibrium equation (\ref{equilib_form}). Recalling that 
$DW(\nabla u^h) (\nabla u^h)^T$ is symmetric (see the proof of 
Lemma \ref{Ehsym}) we obtain:
\begin{equation}\label{r1}
\int_{S^h} f^h\cdot Hu^h = \int_{S^h} DW(\nabla u^h):H\nabla u^h
= \int_{S^h} \Big(DW(\nabla u^h) (\nabla u^h)^T\Big) : H = 0.
\end{equation}
Similarly, taking
$\phi^h = \frac{1}{\epsilon} (\mbox{exp}(\epsilon H)u^h - u^h)$ 
in (\ref{equilib}), by frame indifference of $W$ we get:
$$\int_{S^h} f^h\cdot Hu^h =\lim_{\epsilon\to 0} \frac{1}{\epsilon}
\int_{S^h} f^h \cdot(\mbox{exp}(\epsilon H)u^h - u^h)
= h \lim_{\epsilon\to 0} \frac{1}{\epsilon}
\Big(J^h(\mbox{exp}(\epsilon H)u^h) - J^h(u^h)\Big) = 0.$$
Now, for any sequence of skew-symmetric matrices $F^h$ we have:
\begin{equation}\label{r2}
\begin{split}
 \int_S f\cdot &Q^h F^h V^h = \frac{1}{he^h} \int_{S^h} f^h\cdot
Q^h F^h \Big((Q^h)^Tu^h - c^h - \mbox{id}\Big)\\
& = \frac{1}{he^h} \int_{S^h} f^h\cdot \big( Q^hF^h(Q^h)^T\big) u^h
- \frac{1}{he^h}\int_{S^h} f^h~\mbox{d}z \cdot Q^h F^hc^h
- \frac{1}{he^h}\int_{S^h} f^h\cdot Q^hF^hz~\mbox{d}z\\
& = - \frac{h}{\sqrt{e^h}}\int_{S} f\cdot Q^hF^hx~\mbox{d}x,
\end{split}
\end{equation}
where the first two terms in the second line above vanish by 
taking $H=Q^hF^h(Q^h)^T\in so(3)$ in (\ref{r1}), and by
the normalization of $f^h$.  Passing to the limit with $h\to 0$
in (\ref{r2}), where $F^h=F$, we see that:
$-\int_S f\cdot \bar Q F V = \lim_{h\to 0} h/\sqrt{e^h} 
\int_S f\cdot Q^h Fx~\mbox{d}x$.
This implies (\ref{r0}).

Clearly, (\ref{r0}) is also equivalent to $0=B:\bar Q F = \bar Q^T B:F$
for all $F\in so(3)$, which means exactly that
$\bar Q^T B$ is a symmetric matrix.

To prove the other equivalent formulation of (\ref{r0}), notice that:
$$\int_S f\cdot \bar Q Fx = \int_S \bar Q^T f\cdot Fx =
-c_F\cdot \int_S \bar Q^T f\times x = -c_F\int_S f\times \bar Qx,$$
where $c_F\in\mathbb{R}^3$ is such that $Fx = c_F\times x$ 
for all $x\in\mathbb{R}^3$.
Since there is a one to one correspondence between vectors $c_F$ and skew
matrices $F$, the proof is achieved.
\end{proof}

Define now the set of the rotation equilibria:
$$\mathcal{M}=\{\bar Q\in SO(3); ~ \mbox{skew }(\bar Q^T B) = 0\}.$$
Our goal is to derive the third Euler-Lagrange equation, with respect to
the variations of $\bar Q$ in $\mathcal{M}$. For $\bar Q\in\mathcal{M}$, 
let $F\in so(3)$ be such that:
\begin{equation*}\label{r3}
\bar Q F = \lim_{n\to\infty} \frac{\bar Q_n - \bar Q}{\|\bar Q_n - \bar Q\|},
\end{equation*}
for some $\bar Q_n\in \mathcal{M}$ converging to $\bar Q$. Clearly, the above
implies that:
\begin{equation}\label{r4}
\mbox{skew } (F\bar Q^T B) = 0.
\end{equation}

\begin{lemma}\label{lem_EL3}
Under the hypothesis of Theorem \ref{th-main} or Theorem \ref{th-main2}, 
assume moreover that:
\begin{equation*}
\lim_{h\to 0} \frac{Q^h - \bar Q}{\|Q^h - \bar Q\|} = \bar Q H,
\quad \mbox{ with } \quad \mathrm{skew }(H\bar Q^T B)\neq 0.
\end{equation*}
Then for every $F\in so(3)$ satisfying (\ref{r4}) there holds:
\begin{equation*}\label{EL3}
\int_S f\cdot \bar Q F V~\mathrm{d}x = 0.
\end{equation*}
\end{lemma}
\begin{proof}
We will find a sequence $F^h\in so(3)$, converging to $F$ and such that,
for all $h$:
\begin{equation}\label{r5}
\int_S f\cdot Q^h F^hx ~\mbox{d}x = 0.
\end{equation}
In view of (\ref{r2}) this will prove the lemma.
Existence of such approximating sequence $F^h$ is  guaranteed by the assumed
nondegeneracy condition: $\mbox{skew } (H\bar Q^T B) \neq 0$.

Firstly, notice that for $Q^h\in\mathcal{M}$ one can take $F^h = F$. 
Otherwise, define:
$$F^h = F - \frac{(Q^h)^TB:F}{|\mbox{skew }((Q^h)^TB)|^2} 
~\mbox{skew }((Q^h)^TB).$$
Then:
\begin{equation*}
\begin{split}
\int_S f\cdot Q^h F^hx ~\mbox{d}x & = B: Q^hF^h = (Q^h)^T B: F^h\\
& = (Q^h)^T B: F - \frac{(Q^h)^TB:F}{|\mbox{skew }((Q^h)^TB)|^2} ~
(Q^h)^T B: \mbox{skew }\big((Q^h)^TB\big) = 0,
\end{split}
\end{equation*}
and moreover:
\begin{equation*}
\begin{split}
\lim_{h\to 0} |F^h - F| &= \lim_{h\to 0}\frac{|(Q^h)^TB:F|}
{|\mbox{skew }((Q^h)^TB)|} =  \lim_{h\to 0}\frac{|(Q^h)^TB:F - \bar Q^TB:F|}
{|\mbox{skew }((Q^h)^TB - \bar Q^TB)|}\\
&= \lim_{h\to 0} \left|\left(\frac{Q^h-\bar Q}{\|Q^h-\bar Q\|}\right)^T
  B:F\right| / \left|\mbox{skew }\left(\frac{Q^h-\bar Q}
{\|Q^h-\bar Q\|}\right)^T B\right|
= \frac{|H^T \bar Q^T B: F|}{|\mbox{skew }(H^T \bar Q^T B)|} = 0.
\end{split}
\end{equation*}
The last expression above equals to $0$ 
because of the nullity of its numerator:
$$H^T\bar Q^T B: F = \bar Q^T B: HF = \bar Q^T B: (HF)^T = 
\bar Q^T B: FH = -F\bar Q^TB:H = 0,$$ 
where we have used that $\bar Q^T B$ is symmetric and (\ref{r4}).
\end{proof}

\end{document}